\documentclass[12pt]{amsart}
\usepackage{amssymb}
\usepackage{amsmath}
\def \and{\qquad\text{and}\qquad}

\newcommand{\Id}{{1\hspace{-4pt} 1}}

\newtheorem{proposition}{Proposition}[section]
\newtheorem{theorem}[proposition]{Theorem}

\theoremstyle{definition}
\newtheorem{definition}[proposition]{Definition}
\theoremstyle{remark}

\numberwithin{equation}{section}

\begin{document}

\title[Smooth circle actions on $S^7$]{Smooth circle actions on $S^7$ with unbounded periods and non-linearizable multicentres}

\author{Massimo Villarini}
\footnote[1]{Massimo Villarini, Dipartimento di Scienze Fisiche, Informatiche e Matematiche, via Campi 213/b 41100, Universit\'a di Modena e Reggio Emilia, Modena, Italy

E-mail: massimo.villarini@unimore.it}

\begin{abstract}
We will give an example of a smooth free action of $S^1=U(1)$ on $S^7$ whose orbits have unbounded lenghts (equivalently: unbounded periods). As an application of this example we construct a $C^{\infty}$ vector field $X$, defined in a neighbourhood $U$ of $0 \in \mathbb{R}^8$, such that: $U-\{0\}$ is foliated by closed integral curves of $X$, the differential $DX(0)$ at $0$ defines a $1$-parameter group of nondegenerate rotations and $X$ is {\it not} orbitally equivalent to its linearization. This proves in the $C^{\infty}$ category that the classical Poincar\'e Centre Theorem, true for planar nondegenerate centres, is not generalizable to multicentres.
\end{abstract}

\maketitle

\section{Introduction}

Let $X$, $Y$, be $C^1$ vector fields, respectively defined in the open sets $U$, $V$, and suppose that $q \in U$, $p \in V$ are stationary points. These vector fields define singular foliations of $U$ and $V$, and are orbitally equivalent if these foliations are equivalent. In other words, $X$, $Y$ are orbitally equivalent if there exists a diffeomorphism $\phi : U \rightarrow V$ and a smooth positive function $f:U \rightarrow \mathbb R$ such that $\phi_* (f X)=Y$. If $Y=DX(0)$ the question of orbital equivalence goes under the name of linearizability of a vector field in a neighbourhood of a stationary point. Poincar\'e, Siegel, Sternberg, Arnold, Bruno and many others went into the study of the delicate interplay of metrically generic arithmetic  (diophantine) conditions on the eigenvalues of the differential and the linearizability of the vector field, the so called small divisors problem. On the other hand, prompted by non-generic but important dynamical questions, {\it e.g.} Hamiltonian dynamics, Poincar\'e himself started the study of the following more geometric linearization problem, where diophantine conditions fail.
Let $X$ be a vector field in $U\subset \mathbb{R}^2$, having a stationary point at $0$, defining a foliations by circles of $U-\{ 0 \}$, such that $DX(0)$ is the infinitesimal generator of the standard action of $S^1=SO(2)$ on $\mathbb{R}^2$: $0$ is named {\it nondegenerate centre}.  In \cite{poincare} Poincar\'e proved:
\begin{theorem}\label{centro}(Poincar\'e Centre Theorem)
Let $X$ be a real anlytic vector field in a neighbourhood $U \subset \mathbb R^2$ of a {\it nondegenerate centre} at $0$: then $X$ is linearizable, {\it i.e.} is locally orbitally equivalent to its linearization.
\end{theorem}
Poincar\'e 's proof of this theorem is analytic: a geometric proof due to Moussu \cite{moussu} appeared much later.  For a smooth version see \cite{villarini}.

The question answered by this theorem has an obvious analogous in dimension $2n$, $n>1$.
\begin{definition}
Let $X$ be a $C^1$ vector field defined in a neighbourhood $V$ of $0$, the origin in $\mathbb R^{2n}$, $n>1$. Then $0$ is a {\it nondegenerate multicentre} for $X$ if:
\begin{itemize}
\item[(i)] $X(0)=0$
\item [(ii)] the differential at $0$, $DX(0)$, is the infinitesimal generator of a $1$-parameter group of nondegenerate  rotations
\item [(iii)] there exists a neighbourhood $U$ of $0$, $U \subset V$, such that $U-\{ 0 \}$ is foliated by closed curves which are trajectories of $X$.
\end{itemize}

\end{definition}
The natural question to ask is:

Let $0$ be a nondegenerate multicentre of $X$: is it linearizable?

Affirmative answer to this question was obtained by Urabe and Sibuya \cite{urabesibuya} and later, with a geometric proof and a slight extension of the result, by Brunella and the author \cite{brunellavillarini}, under the hypotheses:
\begin{itemize}
\item[.] $X$ is real analytic
\item[.] the period function of the nondegenerate multicentre, defined in $U-\{ 0 \}$ as the first-return time of each periodic orbit, is {\it bounded}.
\end{itemize}
On the other hand, it is obvious that if $X$ is a smooth vector field which is linearizable near a nondegenerate multicentre, then its period function is locally bounded near the multicentre.

Therefore an example of a nondegenerate multicentre whose period function is not locally bounded near the stationary point is an example of non-linearizable nondegenerate multicentre, too.

\begin{theorem}\label{multicentri}
There exists a $C^{\infty}$ vector field $X$ defined in a neighbourhood $V$ of the origin $0$ in $\mathbb R^8$ such that:
\begin{itemize}
\item [(i)] $0$ is a nondegenerate multicentre for $X$
\item [(ii)] $X$ is not orbitally equivalent to the linear vector field defined by $DX(0)$ in any neighbourhood of $0$: {\it i.e.}, $X$ is not linearizable at $0$.
\end{itemize}
Moreover $X$ has an analytic first integral whose level sets are Euclidean $7$-spheres in $\mathbb R^8$, and the period function of $X$ is {\it unbounded} on each invariant $7$-dimensional sphere with sufficiently small radius.
\end{theorem}

Smoothness of the vector field $X$ in the statement means infinite differentiability. Our example is not real analytic.

The dynamical meaning of Theorem (\ref{multicentri}) can be stated as follows: it is possible to perturb nonlinearly $n \geq 4$ identical harmonic oscillators, keeping that energy is conserved and dynamics are periodic for all initial data sufficiently close to the stationary point, obtaining a truely nonlinear and uncoupled system. This is not possible if $n=1$ (Poincar\'e, Theorem (\ref{centro})) or $n=2$ (D. Epstein \cite{epstein}, Theorem (\ref{essetre}) below).

The construction of the example in Theorem (\ref{multicentri}) is obtained by blowing up the stationary point, and by the following theorem which is perhaps of independent interest:
\begin{theorem}\label{essesette}
There exists a smooth free action of $S^1 = U(1)$ on the Euclidean $S^7$, defined by a smooth vector field $X$, such that its  minimal periods are unbounded. Moreover there exists a smooth $\lambda$-family of vector fields $X_{\lambda} : S^7 \rightarrow T S^7$, $\lambda \in [0,1[$, which smoothly converges as $\lambda \rightarrow 1$ to the Hopf vector field, {\it i.e.} to the isochronous infinitesimal generator of the Hopf circle bundle $S^1 \hookrightarrow  S^7 \rightarrow \mathbb C \mathbb P^3$.

\end{theorem}

We briefly outline how these results are obtained. In the next section we give a rather detailed description, and a slight modification, of a celebrated example by D. Sullivan \cite{sullivan} of a real analytic closed manifold foliated by circles through the trajectories of a smooth vector field whose period function is unbounded. Our adaptation of Sullivan's example is a dynamical system consisting of two rigid bodies constrained to have the same istantaneous rotation axis, having  periodic dynamics, with unbounded period, for any initial condition in their compact phase space. This example is developed in the third section to give the proof of Theorem (\ref{essesette}), and to deduce from it Theorem (\ref{multicentri}). Theorem (\ref{essesette}) follows directly from Theorem (\ref{fibrati}) which shows that there exist two Milnor spheres \cite{milnor}, corresponding to the quaternionic Hopf bundles, which are diffeomorphic to the Euclidean $S^7$ and admit fibrations by circles, based on the Sullivan's vector field, with unbounded period function.

{\it Aknowledgments}: I wish to thank my collegues C. Benassi, G. Mazzuoccolo, G. P. Leonardi for useful discussions, and F. Podest\'a for enlightening explanations.

The questions addressed in this article were brought to the author's attention by the collaboration \cite{brunellavillarini} with Marco Brunella. During the rather long period it took to get these partial results I had many times the pleasure to talk with him about this work. This pleasure has been interrupted by his untimely disappearance: I dedicate this work to his memory.

\noindent

\section{A vector field by D. Sullivan.}
\noindent
Let $P$ be a closed differentiable manifold. The following question arose around $1950$'s in relation to stability theory of foliations (Reeb, Haefliger):

 Let $P$ be smoothly foliated by circles:
 
Is the lenght of the leaves a bounded function on $P$?

The answer to this question is independent of the metric on $P$.

Let $X$ be a vector field on $P$ whose integral curves are the leaves of the foliation by circles of $P$, and let
$$
T:P \rightarrow {\mathbb{R}}^+
$$
be the (minimal) period function of $X$, {\it i.e.} denoting $e^{t X}$ the flow of $X$ let:
$$
e^{T(p) X}(p)=p
$$
for every $p \in P$, while $e^{t X}(p) \neq p$ if $0<t<T(p)$. Then the above question is equivalent to:

 Is $T$ bounded on $P$?

The answer to this question depends only on the foliation by circles of $P$.

The affirmative answer to the above questions has been known as the {\it Periodic Orbits Conjecture}, and it has been rather intensively studied around $1970$'s. The main known results are:

\begin{itemize}
\item [(i)] The conjecture is true if $dim \, P=3$ : this has been proved by D. B. A. Epstein \cite{epstein}, who also proved the following theorem, which should be compared to our Theorem (\ref{essesette}):
\begin{theorem}\label{essetre}(Epstein)
Any smooth free circle action on $S^3$ defines a circle bundle which is isomorphic to the Hopf bundle  $S^1\hookrightarrow S^3 \rightarrow S^2$.
\end{theorem}

\item [(ii)] The conjecture is false if $dim \, P \geq 4$. A smooth counterexample when $dim \, P=5$ has been given by D. Sullivan \cite{sullivan}: in the same article it is also described a real analytic example, due to W. Thurston. Later Epstein and Vogt \cite{epsteinvogt} gave a real analytic counterexample in dimension $4$.
\end{itemize}

This section is devoted to a rather detailed exposition of Sullivan's example outlined in \cite{sullivan}. We felt it necessary due to the importance of this example in the proofs of our main results and also in view of needed modifications of it. On the other hand the  example by Thurston, throughly explained in \cite{sullivan} and in the books \cite{besse}, Appendix A by D. Epstein, and \cite{godbillon}, does not fit well for our applications.

Let $\pi : S(S^2) \rightarrow S^2$ be the unit tangent bundle of the Euclidean $2$-sphere and denote:
$$
 M= S(S^2) \times_{S^2} S(S^2) = \{ (p,p') \in S(S^2) \times S(S^2) : \pi (p)= \pi (p') \}. 
$$
$M$ is the total space of the $T^2$-bundle (fiber product of $S(S^2)$ with itself) $\pi : M \rightarrow S^2$. The representation of a vector field $X:M \rightarrow TM$ as $X=(X_1 , X_2)$ where $X_1 (p,p') \in T_{p}S(S^2)$, $X_2 (p,p') \in T_{p'}S(S^2)$ follows from the embedding $M \hookrightarrow S(S^2) \times S(S^2)$. We will denote:
$$
H : S(S^2) \rightarrow T S(S^2)
$$
the Hopf vector field, {\it i.e.} the infinitesimal generator of the $S^1$-action on the unit tangent bundle: rather explicit models of these mathematical objects will be given in the next section, see (\ref{modello}) and the following remarks.

A smooth family of vector fields $Y_{\lambda} :  M \rightarrow TM$ is {\it infinitely (or exponentially) flat} as $\lambda \rightarrow \infty$ if near any $(p,p') \in M$ there exist local coordinates $u$ on $M$ such that for any multindex $\underline k = ( k_1, \ldots ,k_4)$ there exists two positive constants $c_1 (\underline k)$, $c_2 (\underline k)$ such that:
$$
\Vert \frac{\partial^{\Vert \underline k \Vert}Y_{\lambda}(u)}{\partial u^{\underline k}}     \Vert < c_1(\underline k) e^{- c_2 (\underline k) \lambda}
$$
for any $\lambda >0$. Of course, the fundamental property of an infinitely flat family of vector fields $Y_{\lambda}$, which will be used several times in this article, is that, putting for instance $\mu = \frac{1}{\lambda}$, $\mu \rightarrow Y_{\mu}$ extends smoothly the family of vector fields originally defined in $M\times]0,1[$ to $M\times[0,1[$, with $Y_0 =0$. The following proposition is a slight modification of the result outlined by Sullivan in \cite{sullivan}:
\begin{proposition}\label{campo}
There exists a smooth family of vector fields:
$$
X_{\lambda} : M \rightarrow TM
$$
such that:
\begin{itemize}
\item [$(1)$] for any $\lambda \in \mathbb{}R
^+$ the orbits of $X_{\lambda}$ are closed, with periods going to infinity as $\lambda \rightarrow \infty$
\item [(2)]
\begin{equation}\label{piatto}
X_{\lambda} = (H,0) + Y_{\lambda}
\end{equation}
where $Y_{\lambda}$ is infinitely flat as $\lambda \rightarrow \infty$.

Moreover, through a locally invertible map $p:\mathbb{R}^+ \cup \{ \infty \} \rightarrow S^1$, $p(\lambda)= \theta$, one can push forward $X_{\lambda}$ to get a smooth vector field on a compact manifold:
$$
X:M\times S^1 \rightarrow TM \times TS^1
$$
such that for any $\theta \in S^1$ each $M\times \{ \theta \}$ is fibered by circles by the orbits of $X$, and the periods of these orbits goes to infinity as $\theta \rightarrow \theta^* \in \{ \theta_1, \ldots , \theta_N \} = p^{-1}(\infty)$.
\end{itemize}

\end{proposition}
The manifold $M$ in this proposition is a slight modification, suitable for our future use, of the manifold where the example in \cite{sullivan} was originally defined.
 
The rest of this section is devoted to the construction pertaining the proof of this proposition. It will be given in algorithmic form, divided in several steps. For reader's convenience, we summarize here the plan of the proof.

Firstly, we define in $\mathbb{C}\simeq \mathbb{R}^2$ a closed continuous curve whose curvature $k(q , \lambda)$ depends on a functional parameter  $h(\lambda)$ and $\frac{1}{k}$, togheter with all its derivatives with respect to local coordinates, tend uniformly on the curve to $0$ as $\lambda \rightarrow \infty$: we will say that $\frac{1}{k}$ is {\it infinitely flat}, for short. This curve is mapped on a small equatorial band of $S^2$ by stereographic projection, getting a closed curve on the sphere whose geodesic curvature $k_g$ has still the property that $\frac{1}{k_g}$ is infinitely flat. The curve so obtained is not $C^1$ because at one point one generally has a corner, {\it i.e.} two different (asymptotic) tangent vectors form an angle $\Delta \theta$. By an argument  briefly sketched in Addendendum $1$ in \cite{sullivan}, for which credit is given to N. Kuiper, we define a smooth deformation of it, defined by rotations of fixed axis, given by the corner point on $S^2$, with rotation angle smoothly depending on the point on the curve, and coinciding with $\Delta \theta$ in a left neighbourhood of the corner. This modification provides a $C^1$ curve and does not change the asymptotic property of the geodesic curvature. Finally, a last modification, acting locally in a neighbourhood of the smoothed corner, gives a $C^{\infty}$ curve on the sphere whose reciprocal geodesic curvature is infinitely flat. The curve on $S^2$ is lifted to a curve on $M$ through the geodesic curvature in the first factor, and through constant speed, equal to the reciprocal of the curve lenght, on the second factor. These parameter-depending curve is moved along $M$ by a natural action of $SO(3) \simeq S(S^2)$ and the so obtained curves form the desired foliation by circles tending to the Hopf fibration in the first factor as the parameter tends to infinity.

{\it $1^{st}$ Step}

We consider the family of trochoids depending on the choice of the smooth function $h(\lambda)$:
\begin{equation}\label{trocoide}
\begin{cases}
\rho (t)&= 1 + a(\lambda) -h(\lambda) \cos (t + \frac{\pi}{2}) \\
\theta &= a(\lambda) t -h(\lambda) \sin t
\end{cases}
\end{equation}
where $0<a<h$, $a(\lambda)=g(\lambda)h(\lambda)$, $h(\lambda) = o(1)$, $g(\lambda)=o(1)$ as $\lambda \rightarrow \infty$, $t \in [0,\frac{3 \pi}{a(\lambda)}]$, $\theta = \theta \, mod \, 2 \pi$: $g(\lambda)$ must tend to $0$ in a sufficiently fast way which will be precised in the third step of this construction. The curvature of this family of trochoids is:

\begin{equation}\label{ktr1}
k_{tr}(\rho , \theta , \lambda)=\frac{\dot \rho \ddot \theta - \dot \theta \ddot \rho}{(\dot \rho^2 + \dot \theta^2)^{\frac{3}{2}}} = \frac{1}{h}\frac{1-g \cos (t + \frac{\pi}{2})}{(1 + g^2 -2g \cos (t + \frac{\pi}{2}))^{\frac{3}{2}}}
\end{equation}
hence:
\begin{equation}\label{ktr}
\frac{1}{k_{tr}} = h (1 + \eta_1 (t,\lambda))
\end{equation}
where $\eta_1 (t,\lambda)$ is smooth and $\frac{1}{k_{tr}}$ satisfies:
$$
\vert \frac{\partial^l k_{tr}^{-1}}{\partial t^l}(t, \lambda) \vert < c(l) h(\lambda)
$$
for every $l=0,1,\ldots$, uniformly with respect to $t\in [0,\frac{3 \pi}{a(\lambda)}]$: throughout all this section we will refer at this kind of property saying that {\it $\frac{1}{k_{tr}}$ is infinitely flat as $\lambda \rightarrow 0$}, obviously refferring to a suitable choice of $h(\lambda )$ to be specified later. 

We consider the family of plane curves:
\[
\gamma(\lambda) \, : \, z(t)=\rho (t) e^{i \theta (t)}
\]
$t \in [0, \frac{3 \pi}{a(\lambda)} ]$, $z=x+iy$ and compute their curvature:
\[
k(z,\lambda)=\frac{\dot x \ddot y - \dot y \ddot x}{(\dot x^2 + \dot y^2)^\frac{3}{2}}=k_{tr}(z;\lambda) \xi(z; \lambda) + \chi (z:\lambda)
\]
where
$$
\xi = \rho (\frac{\dot \rho^2 + \dot \theta^2}{\dot \rho^2 + \rho \dot \theta^2})^\frac{3}{2}
$$
and
$$
\chi = \dot \theta ( \frac{\dot \rho}{(\dot \rho^2 + \rho^2 \dot \theta^2)^{\frac{3}{2}}} +\frac{1}{(\dot \rho^2 + \rho^2 \dot \theta^2)^{\frac{1}{2}}} ).
$$
Therefore:
\begin{equation}\label{k}
\frac{1}{k(z,\lambda)}= h(\lambda)(1 + \eta_2 (z;\lambda))
\end{equation}
is smooth and infinitely flat as $\lambda \rightarrow 0$.

{\it $2^{nd}$ Step}

Firstly, we reparametrize the curve $\gamma (\lambda)$ defined in the previous step with a parameter, still called $t$, belonging to $[0,3 \pi]$: this reparametrization has no influence on the curvature, and therefore also on the asymptotic of the reciprocal curvature and on the fact that estimates of type (\ref{k}) hold. $\gamma (\lambda)$ is not a closed curve, in general. To get a closed curve satisfying property (\ref{k}) we define a non local deformation of $\gamma (\lambda)$. Firstly, we observe that the intersections between the trochoids defined by (\ref{trocoide}) and the curve $\rho = 1 + a (\lambda)$ are transverse, and occur at values of the parameter $t$ whose distance is $s(\lambda) \simeq a(\lambda)$, hence $s(\lambda)=o(1)$. Let $t(\lambda) \in ]2 \pi, 2 \pi + s(\lambda)[$ such that $\rho (t(\lambda))= 1 + a(\lambda)$: transversality of the intersections implies that $t \rightarrow t(\lambda)$ is smooth. Let:
$$
(\rho , \theta ) \rightarrow  \Theta (\rho , \theta) = -\alpha (\rho) \beta (\theta) \frac{\partial}{\partial \theta}
$$
be a smooth vector field, and let $0\leq \alpha , \beta \leq 1$ be smooth bump functions, such that:
\[
supp \, \alpha \subset [1+a-2h , 1 +a+ 2h]
\] 
\[
supp \, \beta \subset [\frac{3}{2}\pi , \frac{5}{2} \pi].
\]
Let $(\tau , \rho ,\theta ) \rightarrow (\rho , \zeta(\tau , \rho , \theta )))$ be the flow of $\Theta$, and let $\tau_0 \in ]0, s(\lambda)[$ be the solution of:
$$
\zeta (\tau , 1 + a(\lambda) , \theta (t(\lambda))) = 1 + a(\lambda).
$$
The map:
$$
\phi_{\lambda}(\rho , \theta)=(\rho , \zeta (\tau_0 , \rho , \theta)
$$
transforms the trochoid defined in (\ref{trocoide}) in a closed curve, and the condition $s(\lambda) = o(1)$ implies that denoting:
\begin{equation}\label{bound}
\phi_{\lambda}(\rho , \theta)= (\rho , \theta ) + F(\rho , \theta , \lambda)
\end{equation}
$F$ and all its derivatives with respect to $\theta, \rho$ tend uniformly, with respect to the parameter,  to $0$ as $\lambda \rightarrow \infty$. We still name $\gamma (\lambda)$ the continuos closed curve parametrized by:
$$
t \rightarrow \phi_{\lambda} \circ z (t)
$$
$t \in [0, t(\lambda)]$: it is a smooth curve outside $z(0)$, and we denote:
\begin{equation}\label{angolo}
\Delta \theta (\lambda)= angle \, (lim_{t \rightarrow 0^+} \dot z (t) , lim_{t \rightarrow t(\lambda)^-} \dot z (t))
\end{equation}
the angle between the limit tangent vectors at the ends of the parametrizing interval, which in general is not $0$. Of course:
$$
  \Delta \theta : \mathbb{R}^+ \rightarrow  S^1
$$
is a smooth function of $\lambda$.  From (\ref{ktr}) and (\ref{bound}) it follows that the curvature of such continuos closed curve satisfies  a property of type (\ref{k}), {\it i.e.} its reciprocal is infinitely flat as $\lambda \rightarrow \infty$.

{\it $3^rd$ Step}

Let $\gamma (\lambda)$ be the curve defined in the previous step of the construction, and suppose that it is parametrized by  $t \rightarrow z_{\lambda} (t)$, $t \in [0,2\pi]$. We suppose that $S^2=\{ x \in \mathbb{R}^3 : x^2_1 + x^2_2 + x^2_3 = 1 \}$ and denote $\phi^{-1}$ the stereographic projection from the south pole of the sphere $S^2$ onto the equatorial plane. Let:
$$
\tilde \Gamma_{\lambda} \, : \, t \rightarrow \phi \circ z_{\lambda}(t).
$$
The geodesic curvature $\tilde k_g$ of this curve on $S^2$ is expressed in terms of the Christoffel symbols and the curvature $k$ of $\gamma (\lambda)$ as (see \cite{kreyszig} (49.7) and the comment following that formula):
\begin{equation}\label{christoffel}
\tilde k_g (q,\lambda) = (\Omega (t, \lambda) + 
   k ( \phi^{-1}(p), \lambda ) )   \sqrt{\Delta(\phi (p))} \nonumber
\end{equation}
where $\Delta$ is the determinant of the metric tensor and:

\begin{align*}
\Omega = & ( \Gamma_{11}^2 (p) (\dot x(\phi^{-1}(p)))^2  +  ( 2\Gamma_{12}^2 (p) - \Gamma_{11}^1 (p) ) (\dot x(\phi^{-1}(p)))^2 (\dot y(\phi^{-1}(p))) + \\
 &(\Gamma_{22}^2 (p) - 
 2 \Gamma_{12}^1 (p))    \dot x(\phi^{-1}(p)) (\dot y(\phi^{-1}(p)))^2 
  - \Gamma_{22}^1 (p) (\dot y(\phi^{-1}(p)))^2) t^3 \sqrt{\Delta(\phi (p))} .
\end{align*}

 When $\lambda$ is big enough, the image of $\gamma (\lambda)$ through $\phi$ lay on a very thin equatorial band on the sphere, and the Christoffell symbols are bounded togheter with all their derivatives (actually in this case they tend to $0$ uniformly with respect to the point on the sphere). Moreover $\Delta(\phi (p))=1+o(1)$. The derivatives of the paramatrization of $\gamma (\lambda)$ are bounded if $g(\lambda)$ in (\ref{trocoide}) tends to $0$ sufficiently fast, as we will suppose: in fact, with a suitable choice of $g(\lambda)$ the curve $\gamma (\lambda)$ is approximatly $2 \pi$ long, and its tangent vector rotates in an approximately uniform way. Summarizing, there exsists $C>0$ such that:
\begin{equation}\label{geodetica}
\frac{1}{\tilde k_g (q(t),\lambda)} = C(h(\lambda) + \eta_3 (t, \lambda ))
\end{equation}
where $q$ is a point on the curve, $\eta_3 (t, \lambda )$ is smooth, and $\frac{1}{\tilde k_g (q(t),\lambda)}$ is infinitely flat as $\lambda \rightarrow \infty$ uniformly with respect to $t$. In this way, we obtained a closed continuous curve, smooth everywhere except for a discontinuity of the unit tangent vectors at a point, where unit tangent vectors form the angle $\Delta \theta (\lambda)$. To get a closed $C^1$ curve on $S^2$ still satisfying (\ref{geodetica}) we develop the argument due to N. Kuiper briefly outlined in \cite{sullivan}, Appendix $1$. Let:
$$
t \rightarrow \eta_{\lambda} (t)
$$
$t\in[0,2\pi]$, be a parametrization of $\tilde \Gamma_{\lambda}$. We suppose that the discontinuity of the tangent vector to $\tilde \Gamma_{\lambda}$ occurs at $\eta_{\lambda} (0)$. Let $R_{\lambda}(t) \in SO(3)$ be a rotation of axis $\eta_{\lambda} (0)$ and angle given by the smooth function:
$$
\alpha : [0,2 \pi] \times \mathbb{R}^+ \rightarrow S^1
$$ 
where:
$$
\alpha (t,\lambda)= \delta(t) \Delta(\theta (\lambda))
$$
and $\delta : [0, 2 \pi] \rightarrow [0,1]$ is a smooth monotone function, $\delta(t) \equiv 0$ when $t\in[0,\pi]$, $\delta(t) \equiv 1$ when $t\in[\frac{3}{2} \pi, 2\pi]$. Of course, $\alpha$ and all its derivatives are bounded. The curve:
$$
\Gamma_{\lambda} \, : t \rightarrow R_{\lambda} \phi \circ z_{\lambda}(t)= \xi_{\lambda}(t).
$$
is a closed  $C^1$ curve immersed in $S^2$. In fact, continuity of the tangent unit vectors at the point $\xi_{\lambda}(0)$ follows from:
$$
\frac{d}{dt}_{\vert  t= 2 \pi^-}\xi_{\lambda}= R_{\lambda} (\Delta \theta(\lambda)) \dot z_{\lambda}(2\pi)=\dot z_{\lambda}(0). 
$$ 
To get an estimate on the geodesic curvature of $\Gamma_{\lambda}$, we use again (\ref{christoffel}) applied to a parametrization of $\Gamma_{\lambda}$ by $t \in [0, 2 \pi]$ and referred to suitable choices of local coordinates on the sphere: though not necessarily small as for $\tilde \Gamma_{\lambda}$, the Christoffel symbols appearing in (\ref{christoffel}) are bounded independently of $\lambda$, and therefore we get the asymptotic behaviour of the geodesic curvature:
\begin{equation}\label{geodetica1}
 \frac{1}{k_g (q,\lambda)} = C (h(\lambda) + \eta_4(t,\lambda))
\end{equation}
where as usual $\frac{1}{k_g (q,\lambda)}$ is smooth and infinitely flat.
We have obtained a $C^1$ curve on the sphere, smooth everywhere except at a point, with geodesic curvature satisfying (\ref{geodetica1}): a local deformation of this curve transforms it in a smooth curve on the sphere whose geodesic curvature still satisfies (\ref{geodetica1}). The argument consists in smoothing a curve defined as the graph of a $C^1$ function $u \rightarrow f(u)$, $\vert u \vert <1$, smooth outside $u=0$: curvature jumps at $u=0$ from $f''(0^-)$ to $f''(0^+)$, and we modify $f(u)$ in a small neighbourhood $[-a(\lambda),a(\lambda)]$ of $0$, $a ( \lambda ) \rightarrow 0$ and smooth, in two steps: firstly we interpolate linearly between $f(u(-a(\lambda))$ and $f(u(a(\lambda))$, and then we smooth the corners with an arbitrarily small change in the second derivatives. In this way we obtain the desired smooth curve, still satisfying (\ref{geodetica1}): 

{\it $5^{th}$ Step}

So far we have defined a smooth $\lambda$-family of curves $\Gamma_{\lambda}$ on $S^2$, having lenghts $l(\lambda)$ tending to infinity when $\lambda \rightarrow \infty$, with reciprocal of the geodesics curvature which is infinitely flat for $\lambda \rightarrow \infty$. We show now, following an argument in \cite{sullivan}, that these curves define a smooth $\lambda$-family of foliations of $M$ tending as $\lambda \rightarrow \infty$ to the foliation on $M$ defined by the vector field $(H,0)$, with asymptotic behaviour depending on the choice of the functional parameter $h(\lambda)$. Let $(p,p')\in M$: then $\pi (p)=\pi (p')=q$ and we will denote, with slightly redundant notation, $(p,p')=(q,v,w)$ where $v,w$ are unit tangent vector to the sphere. Recalling that $SO(3) \simeq S(S^2)$ we define the action of $SO(3)$ on $M$:
\begin{equation}\label{azione}
g * (q,v,w)=(g*q,dg(q)v, dg(q)w)
\end{equation}
where $g * q$ is the standard action by left multiplication of $SO(3)$ onto itself.
The curve $\Gamma_{\lambda}$ on $S^2$ is lifted to a closed curve $\Gamma_{\lambda}' \subset S(S^2)$ associating to every $q \in \Gamma_{\lambda}$ its unit tangent vector. The curve $\Gamma_{\lambda}'$ is lifted to a closed curve $\Gamma_{\lambda}''$ on $M$ in the following way. Let $p_0 \in \Gamma_{\lambda}$ be fixed. To any $p=(q,v) \in \Gamma_{\lambda}$ we associate the lifted point $(p,p') \in M$ such that $p'=(q,w)$ where:
\begin{equation}\label{angolo}
w=e^{i 2 \pi \frac{l(p,\lambda)}{l(\lambda)}} v.
\end{equation}
Here $l(\lambda)$ is the lenght of $\Gamma_{\lambda}$ and $l(p,\lambda)$ is the lenght of the arc in $\Gamma_{\lambda}$ of extrema $p_0$, $p$. In other words, we add to the curve $\Gamma_{\lambda}$ a unit tangent vector rotating with constant speed $\frac{1}{l(\lambda )}$ once while the movable point goes on the curve $\Gamma_{\lambda}$. Let:
$$
\Gamma_{\lambda , g} = g*\Gamma_{\lambda}''.
$$ 
The closed curves $\Gamma_{\lambda , g}$, $g \in SO(3)$, define a smooth foliation by circles of $M$. To prove this claim we must show that:
\begin{itemize}
\item[$1$.] if $g,h \in SO(3)$, $g\neq h$, then $\Gamma_{\lambda , g} \cap \Gamma_{\lambda , h} = \emptyset$
\item[$2$.] for any $(q,v,w) \in M$ there exists $g \in SO(3)$ such that $(q,v,w) \in \Gamma_{\lambda , g}$.
\end{itemize}
To prove the first property is equivalent to show that:
$$
\Gamma_{\lambda , g} \cap \Gamma_{\lambda }'' = \emptyset
$$
for every $g \in SO(3)$, $g \neq \Id$. Let $(q,v,w) \in \Gamma_{\lambda , g} \cap \Gamma_{\lambda }''$: then there exists $(q',v',w') \in  \Gamma_{\lambda }''$ such that $g *(q',v',w')=(q,v,w)$. Using that $g$ is an isometry, we have that the angles $<v,w>$, $<v',w'>$ are equal, hence:
$$
e^{i 2 \pi \frac{l(q,v,\lambda)}{l(\lambda)}} = e^{i  2 \pi \frac{l(q',v',\lambda)}{l(\lambda)}}
$$
which implies that $g=\Id$. To prove $2$ we observe that the standard action of $SO(3)$ onto itself is transitive hence for every $(q',v')\in \Gamma_{\lambda}''$ there exists a $g \in SO(3)$ such that $g*(q',v')=(q,v)$. Among the elements of $\Gamma_{\lambda}''$ there exists exactly one $(q',v',w')$ such that $g*(q',v')=(q,v)$ and $w'=e^{i 2 \pi \frac{l((q,v),\lambda)}{l(\lambda)}} w$. Summarizing, we proved that the $\Gamma_{\lambda , g}$, $g\in SO(3)$, $\lambda \in \mathbb{R}^+$, define a smooth family of fibrations by circles of $M$. The leaves of this fibrations are the integral curves of the vector fields $Z_{\lambda} : M \rightarrow TM$, $Z_{\lambda}=(Z_{\lambda , 1},Z_{\lambda , 2})$, defined, with local notation, by:
\[
d \pi (q,v,w) Z_{\lambda , 1} (q,v,w)=d \pi (q,v,w) Z_{\lambda , 2} (q,v,w) = v
\]
and:
\[
\begin{cases}
Z_{\lambda , 1} (q,v,w) = & ( v , k_g (q,v,\lambda) H(q,v))\\
Z_{\lambda , 2 } (q,v,w) = & ( v , \frac{1}{l(\lambda)} H(q,v)).
\end{cases}
\]
On the other hand, the scalar $k_g (q,v,\lambda)$ is globally defined on $M$ and therefore:

$$
X_{\lambda}(q,v,w)=\frac{1}{k_g (q,v,\lambda)} Z_{\lambda}(q,v,w)
$$
is a smooth family of vector fields having the properties specified in Proposition (\ref{campo}). In fact, $X_{\lambda}$ decomposes locally as:
$$
X_{\lambda}=(X_{\lambda , 1},X_{\lambda , 2})
$$
where:
$$X_{\lambda , 1}(q,v,w)=(\frac{1}{k_g (q,v,\lambda)}v,H(q,v))
$$
and:
$$
X_{\lambda , 2}(q,v,w)=(\frac{1}{k_g (q,v,\lambda)}v , \frac{1}{k_g (q,v,\lambda)} \frac{1}{l(\lambda}).
$$
Let $X_{\lambda} = ((0,H),(0,0)) + \tilde X_{\lambda}$: we are going to show that $\tilde X_{\lambda}$ is infinitely {\it exponentially} flat as $\lambda \rightarrow \infty$ if we choose: 
$$
h(\lambda)=e^{- \lambda}.
$$ 
Firstly, the fact that:
$$
X_{\lambda} \rightarrow (0,H)
$$
as $\lambda \rightarrow \infty$ follows from the above local definition of $X_{\lambda}$. Moreover, we observe that $X_{\lambda}$ and $((0,H),(0,0))$ are both invariant with respect to the action (\ref{azione}), hence $\tilde X_{\lambda}$ is invariant, too. Invariance of $X_{\lambda}$ follows from the definition of its integral curves as $SO(3)$-translations of $\Gamma_{\lambda}''$, while $SO(3)$-invariance of the Hopf vector field on the first factor of $M$ follows from the well known property of congruence of Hopf circles. A proof of this last property can be obtained, for instance, from the expression of the Hopf vector field in (\ref{modello}) and the fact that the wedge product $\wedge$ in $\mathbb{R}^3$ is invariant with respect to $SO(3)$. Therefore to prove exponential flatness of $\tilde X_{\lambda}$ we just need to prove it with respect to derivations relative to a parameter along the curve $\Gamma_{\lambda}$, and this is what we proved in (\ref{geodetica1}).

To end the proof of Proposition (\ref{campo}) it remains to push forward this family of vector fields on $M \times S^1$ through a suitable choice of a map $p:\mathbb{R}^+ \cup \{ \infty \} \rightarrow S^1$: this part of the proof is similar to, and simpler than, the analogous one in Thurston's example. As we will never use it in the rest of this article, we leave it to the reader.

\section{Proofs of the main results}

To help geometric insight we introduce a well-known model of $\pi : S(S^2) \rightarrow S^2$. Let: 
$$
S(S^2) = \{ (x,y)\in \mathbb R_x^3 \times \mathbb R_y^3 \, : \, \vert x \vert^2 = \vert y \vert^2 = 1, x \cdot y =0 \}
$$
where $x \cdot y =x_1 y_1+ x_2 y_2 + x_3 y_3$. Let: $\pi (x,y)=x$ and define the Hopf vector field:
\begin{equation}\label{modello}
H :
\begin{cases}
\dot x = & 0 \\
\dot y =   & A(x) y
\end{cases}
\end{equation}
where:
\begin{equation}
A(x)=
\left(
\begin{array}{ccc}
0 & -x_3 & x_2  \\
x_3 & 0 & -x_1  \\
-x_2 & x_1 & 0 
\end{array} \right)
\end{equation} 
hence $A(x) y = y \wedge x$, $\wedge$ being the vector product in $\mathbb R^3$. Then $\pi : S(S^2) \rightarrow S^2$ with $S^1$-action defined by $(t,(x,y))\rightarrow e^{t A(x)}(x,y)$ is the ( Hopf ) unit tangent bundle of the $2$-sphere.

The above model generalize to $\pi :M=S(S^2) \times_{S^2} S(S^2) \rightarrow S^2$ in the following way. Let:
$$
M = S(S^2) \times_{S^2} S(S^2) = \{ (x,y,z)\in \mathbb R_x^3 \times \mathbb R_y^3 \times \mathbb R^3_{y'} \, : \, \vert x \vert^2 = \vert y \vert^2 =\vert y' \vert^2= 1, x \cdot y = x \cdot y' = 0 \}
$$
\[
\begin{cases}
\pi :  M  \rightarrow & S(S^2) \\
\pi (x,y,z) = & x
\end{cases}
\]
\[
(H,H) :
\begin{cases}
\dot x = & 0 \\
\dot y = &  A(x) y \\
\dot y' = &  A(x) y'.
\end{cases}
\]
The vector fields $X_{\lambda}$ defined in the previous section can be seen, for large $\lambda$, as the following perturbation of $(H,H)$:
\[
 X_{\lambda}:
\begin{cases}
\dot x = & a(x,y,y',\lambda) \\
\dot y = &  A(x) y + b(x,y,y',\lambda) \\
\dot y' = & c(x,y,y',\lambda)
\end{cases}
\]
where $a(x,y,y',\lambda)$, $b(x,y,y',\lambda)$, $c(x,y,y',\lambda)$ are exponentially flat smooth functions for $\lambda \rightarrow 0$.  

We can embed $M \hookrightarrow S(S^2) \times S(S^2)$ as a subbundle of the cartesian product bundle $\pi \times \pi : S(S^2) \times S(S^2) \rightarrow S^2 \times S^2$ obtaining the family of vector fields, still named $X_{\lambda}$:

\[
X_{\lambda} :
\begin{cases}
\dot x = & a(x,y,y',\lambda) \\
\dot x' =& 0\\
\dot y = &  A(x) y + b(x,y,y',\lambda)\\
\dot y' = &  c(x,y,y',\lambda)
\end{cases}
\]
The slice $\{(x,x',y,y')\in S(S^2) \times S(S^2): x' = x'^0 \}$,  $x'^0 \in S^2$, is diffeomorphic to $M$ and invariant with repect to the above vector fields, the dynamics on the slices being those of Sullivan's example. Using the double covering $SU(2) \simeq S^3 \rightarrow SO(3) \simeq S(S^2)$ we can lift this smooth $\lambda$-family of vector fiels to a still named $X_{\lambda}$ family, defined on $S^3 \times S^3$: {\it this is the meaning that $X_{\lambda}$ will have henceforth. Moreover we suppose that as $\lambda \rightarrow \infty$ the dynamics tend to the Hopf vector field on the}  second {\it factor }.

Let us consider the family $\mathcal S$ of $S^3$-bundles over $S^4$ having $SO(4)$ as structural group: the equivalence classes of these bundles are in one-one correspondence with the elements of $\pi_3((SO(4)) = \mathbb{Z}\oplus \mathbb{Z}$, see \cite{steenrod} and (\ref{classificazione}) below, and therefore are denoted $\xi_{h,j}$, $(h,j) \in \mathbb{Z} \oplus \mathbb{Z}$. The $\xi_{h,j}$'s are defined as follows. Stereographic projections of $S^4=\{ u_1^2 + \cdots +u_5^2 = 1 \}$ from north and south pole, $N=(0, \ldots ,1)$, $S=(0, \ldots ,-1)$, onto $\mathbb{R}^4$ with coordinates $v_i=u_i$, $i=1,\ldots ,4$, gives $S^4 = (\mathbb{R}^4 \coprod \mathbb{R}^4)/\simeq$, where $v \simeq \frac{v}{\Vert v \Vert^2}$, $v \neq 0$.  We will make free use of the identification $\mathbb{C}^2 \simeq \mathbb{R}^4 \simeq \mathbb{H}$, where $\mathbb{H}$ is the quaternion algebra, and  $\mathbb{H}_1$ denotes the multiplicative group of quaternions of modulus one. The fibers of the bundles are diffeomorphic to:
\begin{equation}\label{3sfera}
S^3 =\{z=(z_1 , z_2) \in \mathbb{C}^2: \vert z_1 \vert^2 + \vert z_2 \vert^2 =1 \} \simeq \mathbb{H}_1 \simeq SU(2)
\end{equation}
therefore we will denote: $q \in \mathbb{H}_1$, $z=(z_1 , z_2) \in S^3$ and with abuse of language we will identify $q=(z_1 , z_2)$ meaning that:
\[
M(q)=
\left(
\begin{array}{cc}
z_1 & z_2   \\
- \overline z_2 & \overline z_1 
\end{array} \right)
\]
or that $q=a+bi+cj+dk$, $z_1=a+bi$, $z_2=c+di$.

We define the fiber bundles:
\begin{equation}\label{milnor}
\xi_{h,j} = ((\mathbb{R}^4 \times S^3) \coprod (\mathbb{R}^4 \times S^3))/\simeq
\end{equation}
where, denoting an element of $S^3 \simeq \mathbb{H}_1$ as the quaternion of unitary modulus $w$, the transiction function $\varphi_{SN , hj}(v, w))$ between the two trivializng charts of the fiber bundle is given by:
\begin{equation}\label{transizione1}
(v,w)\simeq (\frac{v}{\Vert v \Vert^2}, \frac{v^h w v^j}{\Vert v \Vert}) = (\frac{v}{\Vert v \Vert^2}, \varphi_{SN , hj}(v, w))
\end{equation}
where $h+j=1$.
Restricting the transiction functions to the equator $\{v\in \mathbb{R}^4 : \Vert v \Vert =1 \}$, one obtains the maps:
\begin{equation}\label{classificazione}
\begin{cases}
\varphi_{SN,hj} : S^3 \rightarrow SO(4)\\
\varphi_{SN , hj}(v, w) = v^h w v^j
\end{cases}
\end{equation}
where quaternion multiplication is understood: these maps classify $\xi_{h,j}$, for integers $h,j$, $h+j=1$ \cite{steenrod}. In \cite{milnor} Milnor proved that the total space $E_{h,j}$ of each $\xi_{h,j}$ is homeomorphic to the Euclidean $S^7$, but only the $E_{h,j}$ such that $k^2 -1 \equiv 0 \, mod \, 7$, $k=h-j$, are diffeomorphic to the Euclidean $S^7$.

$\xi_{1,0}$, $\xi_{0,1}$ are the only principal $SU(2)$-bundles in $\mathcal S$ \cite{taubes}: they are not isomorphic as principal bundles, differing for the orientation of the fibers \cite{milnor}. In fact, both $\xi_{1,0}$, $\xi_{0,1}$ are quaternionic Hopf bundles obtained from the restriction of the tautological bundle $\mathbb{H}\times \mathbb{H} -\{(0,0)\} / \mathbb{H}^* \rightarrow \mathbb{HP}$ to the Euclidean unit sphere $S^7 \subset \mathbb{R}^8 \simeq  \mathbb{H}\times \mathbb{H}$ and the restriction of  the action of $\mathbb{H}^*=\mathbb{H}-\{0\}$ to $\mathbb{H}_1$: $\xi_{1,0}$ and $\xi_{0,1}$ are obtained respectively from {\it left} and {\it right} action of $\mathbb{H}^*$. These remarks imply that $E_{1,0}$, $E_{0,1}$ are both diffeomorphic to the Euclidean $S^7$. Let us enter in some details concerning the definition of these bundles, paying attention to their compatibility with a suitable notion of Hopf fibration $S^1 \hookrightarrow S^7 \rightarrow \mathbb{CP}^3$.

For instance, $\xi_{0,1}$ is obtained from the intersections with $S^7$ of the {\it right} quaternionic lines through the origin of $\mathbb{H} \times \mathbb{H}$ given by:
\begin{equation}\label{destra}
q_2 = q_1 v
\end{equation}
where $v$ is a quaternion variable parametrizing $U_N$: these intersections are the fibers $\pi^{-1} (v) \simeq S^3$, parametrized by the unit quaternions $w$. Let $S^1 \hookrightarrow S^7 \rightarrow \mathbb{CP}^3$ be the Hopf fibration generated by the vector field:
\[
H_{0,1} :
\begin{cases}
\dot z_1 =& i z_1 \\
\dot z_2 =& iz_2 \\
\dot z_3 =& iz_3 \\
\dot z_4 =&iz_4 
\end{cases}
\]
and let $q_1 =(z_1 , z_2)$, $q_2 =(z_3 , z_4)$. The fibers of $\xi_{0,1}$ are invariant with respect to $H_{0,1}$, for from (\ref{destra}):
$$
e^{it} q_2 = e^{it} q_1 v
$$
therefore in trivializing coordinates $H_{0,1}$ defines on each fiber the left action of $S^1$ on $S^3$:
$$
(e^{it} , w ) \rightarrow e^{it} w 
$$
{\it i.e.} on each fiber $\xi_{0,1}$ $H_{0,1}$ define the Hopf vector field on $S^3$.
The trivial identity:
$$
(e^{it} w) \frac{\mu}{\vert \mu \vert} = e^{it} (w \frac{\mu}{\vert \mu \vert})
$$
and (\ref{transizione1}), (\ref{destra} ) imply that the transiction function $\varphi_{SN} = \varphi_{SN,0,1}$ of this bundle commutes with the flow of $H_{0,1}$:
\begin{equation}\label{commutazione}
\varphi_{SN} \circ e^{t H_{0,1}} = e^{t H_{0,1}} \circ \varphi_{SN}.
\end{equation}
We syntethize this argument saying that $\xi_{0,1}$ and the $S^1$-bundle defined by $H_{0,1}$ are {\it compatible}. In the same way, $\xi_{1,0}$, defined by the intersections of $S^7$ with the {\it left} quaternion lines:
$$
q_2 = v q_1
$$
turns to be compatible with the bundle $S^1 \hookrightarrow S^7 \rightarrow \mathbb{CP}^3$ defined by the vector field:
\[
H_{1,0} :
\begin{cases}
\dot z_1 =& i z_1 \\
\dot z_2 =& -iz_2 \\
\dot z_3 =& iz_3 \\
\dot z_4 =& -iz_4 
\end{cases}
\]
A differentiable conjugacy between $H_{0,1}$ and $H_{0,1}$ is obtained by the trasformation $(z_1 , z_2 , z_3 , z_4) \rightarrow (z_1 ,\overline z_2 , z_3 , \overline z_4)$.

\begin{theorem}\label{fibrati}
In $E_{0,1}$ there exists a smooth vector field:
$$
X_{0,1}: E_{0,1} \rightarrow T E_{0,1}
$$
whose integral curves define a foliation by circles of $E_{0,1}$, with unbounded periods: an analogous construction holds in $E_{1,0}$.
\end{theorem}
\begin{proof}
We carry on the construction of the vector fields in the statement referring to a general bundle $\xi_{0,1} \in \mathcal{S}$, specializing to the cases of $\xi_{0,1}$, $\xi_{1,0}$ only when needed, see the comment after this proof.

Let $(U_N \times S^3, \psi_N)$ and $(U_S \times S^3 , \psi_S)$, $U_N,U_S \simeq R^4$, be the two trivializing charts defining $\xi_{h,j}$, $(h,j) \in \mathbb{Z} \oplus \mathbb{Z}$, $h+j=1$. The transition function between these charts is $\psi^{-1}_S \circ \psi_N=(\frac{v}{\Vert v \Vert}, \varphi_{SN}(v,w))$, where we recall that:
$$
\varphi_{SN}(v,w)= \frac{w^h v w^j}{\Vert w \Vert}.
$$
Let:
\begin{equation}\label{landa}
\lambda (u_5) = \frac{1}{1-u_5^2}
\end{equation}
and consider the diffeomorphism:
$$
h : S^3 \times ]-1 , 1[ \times S^3 \rightarrow U_N - \{ N \} \times S^3.
$$
such that (spherical coordinates):
\begin{equation}\label{acca}
h^{-1}(v,w)= (\frac{v}{\Vert v \Vert} , {\Vert v \Vert} , w) = (\theta , r ,w)
\end{equation}
where $\theta \in S^3$, $r>0$.
The $\lambda$-depending family of Sullivan's vector fields $X_{\lambda}$ on $S^3 \times S^3$ introduced at the beginning of this section and (\ref{landa})  define a vector field $X$ on the cylinder $S^3 \times ]-1 , 1[ \times S^3$, having each $\{ u_5 = constant \} \simeq S^3 \times S^3$ as an invariant manifold where $X=X_{\lambda (u_5)}$. The diffeomorphism $h$ sends this vector field to $Y = h_* X$ which is well defined in $U_N -\{ N \} \times S^3$ and has the form:
\[
Y :
\begin{cases}
\dot v = &Y_{B }(v,w) \\
\dot w =& Y_{F }(v,w,)
\end{cases}
\]
where $v$ is a coordinate on the first factor in $U_N \times S^3$, the base, with $v(N)=0$, and $w$ is a coordinate on the second factor, the fiber, and $Y_B$, $Y_F$ are respectively the base and fiber component of $Y$.

Firstly, we extend $Y$ to a smooth vector field on $U_N \times S^3$,  such that $Y_B(0,w)=0$, $Y_F(0,w)=H(w)$, $H$ being the Hopf vector field on $S^3$, defining on each fiber the flow:
$$
(t, w) \rightarrow e^{it}w.
$$

In spherical coordinates defined by $h^{-1}$, (\ref{landa}) the invariance of $X_{\lambda (u_5)}$ on $\{  u_5 \, = \, constant \}$ implies that:
\[
Y:
\begin{cases}
\dot \theta = & Y_s(\theta , r) \\
\dot r = & 0 \\
\dot w = & H(w) + o(1)
\end{cases}
\]
where $o(1)$, $Y_s$ are smooth and exponentially flat as $r \rightarrow 0$ as a consequence of the analogous property of the Sullivan's vector field, which differs from the Hopf vector field on the second factor of $S^3 \times S^3$ for an exponentially flat vector field. This property implies the extendibility up to the north polar fiber of $Y$. In fact, going back to $(v,w)$-coordinates, the equation $\dot r =0$ and the asymptotic property of $Y_s$ near the north polar fiber imply that:
$$
\dot v = r \dot \theta =  r Y_s (\theta , r)= \vert v \vert Y_s (\frac{v}{\vert v \vert} , \vert v \vert )=Y_B (v,w)
$$
while the vertical component $Y_F$ remains unchanged. The vector field $Y_r (\theta , r)$ satisfies:
$$
lim_{v \rightarrow 0} \frac{\partial^{\vert \underline l \vert}}{\partial v^{\underline l}}(\vert v \vert Y_r (\frac{v}{\vert v \vert} , \vert v \vert)) =0
$$
for any multindex $\underline l$, which implies that $Y$ extends smoothly at the north polar fiber as the Hopf vector field on $S^3$.

We still name $Y$ the vector field obtained from this extension, and rewrite it in $U_N \times S^3$ as: 
\begin{equation}\label{nord}
Y:
\begin{cases}
\dot v = & Y_B (v, w) \\
\dot w = & Y_F (v,w)= H(w) + \tilde Y_F(v,w)
\end{cases}
\end{equation}
where $Y_B (v,w) , \tilde Y_F(v,w) $ are exponentially flat as $v \rightarrow 0$. The expression of this vector field in $U_S - \{S \} \times S^3$ is obtained changing variable in (\ref{nord}) according to (\ref{transizione}). Let $(u,\eta)$ be trivializing coordinates in $U_S \times S^3$, then:
\begin{equation}\label{nordsud}
\begin{cases}
u = & u(v) =  \frac{v}{\Vert v \Vert} \\
\eta = & \varphi_{SN}(v,w)=\frac{v^h w v^j}{\Vert v \Vert}.
\end{cases}
\end{equation}
and the expression of $Y$ in these coordinates is:

\[
Y:
\begin{cases}
\dot u = & \frac{\partial u}{\partial v}(v(u))Y_{B}(v(u), w(u,\eta)) \\
\dot \eta = &  \frac{\partial \varphi_{SN}}{\partial v}(v(u) , w(u,\eta) ) Y_{B} (v(u) , w(u,\eta))  +  \frac{\partial \varphi_{SN}}{\partial w}(v(u) , w(u,\eta) ) (H( w(u,\eta)) \\ & + \tilde Y_{F} (v(u) , w(u,\eta))).
\end{cases}
\]
The differentials of the change of coordinates are singular at the south polar fiber, their norms tending polynomially at infinity as $u \rightarrow 0$: on the other hand, $u \rightarrow 0$ is equivalent to $u_5 \rightarrow \pm 1$ which implies that the horizontal part $Y_B$ and the vertical term $\tilde Y_F$ go exponentially fast to $0$, leading to the following asymptotic expression for $Y$:

\begin{equation}\label{limite}
\begin{cases}
\dot u =& 0 \\
\dot \eta =& lim_{u \rightarrow 0} \frac{\partial \varphi_{SN}}{\partial w}(v(u) , w (u , \eta) ) H( w (u , \eta)).
\end{cases}
\end{equation}
where existence of the limit  and its equality with the Hopf vector must be proved.

Here we turn to the case of $\xi_{0,1}$ bundle: putting $H=H_{0,1}$ in (\ref{commutazione}) we get:
$$
\frac{\partial \varphi_{SN}}{\partial w}(v(u) , w (u , \eta) ) H(  w (u , \eta)) = H(  w (u , \eta))
$$
and (\ref{limite}) follows.

The fact that, in this case, on the south polar fiber the extended vector field coincides with the Hopf vector field follows from (\ref{commutazione}), too. Summarizing, we obtained a smooth vector field on $E_{0,1} \simeq S^7$ whose integral curves have unbounded periods.

An analogous construction of a fibration by circles on $E_{1,0}$ with unbounded lenght of the leaves, is obtained adapting the above arguments to the case of left $\mathbb{H}^*$ action on the tautological quaternionic bundle: here the main change consists in defining the Hopf fibration $S^1 \hookrightarrow S^7 \rightarrow \mathbb{CP}^3$ through the infinitesimal generator $H_{1,0}$, and consistently defining the $\mathbb{H}_1$ action on $S^3$ as:
$$
(t,w) \rightarrow w e^{it}.
$$
All the other arguments remain unchanged.
\end{proof}
One could wonder if a construction similar to the one just described for $\xi_{0,1}$, $\xi_{1,0}$ can be extended to the other bundles of the family $\mathcal S$: here the main obstruction is (\ref{commutazione}), which probably asks for a suitable definition for every $(h,j)$ of a Hopf-type vector field, defining a $S^1$- action on $E_{h,j}$ which is compatible with the $\xi_{h,j}$ bundle structure.

The proof of Theorem (\ref{essesette}) follows immediatly:
\begin{proof} (of Theorem (\ref{essesette}))
Just apply the previous theorem to the case of $E_{0,1}$ and recall that it is diffeomorphic to $S^7$. The last claim in the statement of the theorem, concerning the existence of a $\lambda$-family of vector fields on the $7$-sphere converging to the Hopf vector field, will be proved, with explicit mention, along the following proof of Theorem (\ref{multicentri}).
\end{proof}
We can prove now the result on multicentres:
\begin{proof}(of Theorem (\ref{multicentri}))
Referring to the proof of Theorem (\ref{fibrati}), we consider on $S^3 \times ]-1,1[ \times S^3 \times \mathbb{R}^+$ the Sullivan's vector field $X_{\lambda}$, with:
\begin{equation}\label{landa1}
\lambda (u_5 , R) = \frac{R}{1-u_5^2}.
\end{equation}
where $u_5 \in ]-1,1[$ and $R \in \mathbb{R}^+$.
For any fixed $R$ let:
$$
\tilde h (v,u_5 , w ,R) = h( v , u_5 , w)
$$
where $h$ is the diffeomorphism defined in (\ref{acca}) and $v$, $u_5$, $w$ are coordinates on $S^3 \times ]-1 , 1[ \times S^3$. Let $Y=   h_* X_{\lambda (u_5 , R)}$  defined on the parallels $\{ u_5 = constant \} \simeq S^3 \times S^3$: the proof of Theorem (\ref{fibrati}) shows that for every fixed $R$, a smooth extension of this vector field defines a fibration by circles on $ S^7$ with unbounded period function. Let $(r,z)$ be  spherical coordinates in a neighbourhood of $0 \in \mathbb{R}^8$, related to $x$-coordinates on $\mathbb{R}^8$ by:
\[
\begin{cases}
r &= \Vert x \Vert \\
z &= \frac{x}{\Vert x \Vert}
\end{cases}
\]
and let $r=\frac{1}{R}$. Writing $Y$ in spherical coordinates:
\[
\begin{cases}
\dot r &= 0 \\
\dot z &= Y(r,z).
\end{cases}
\]
From the definition of Sullivan's vector field, when $r \rightarrow 0$:
$$Y(r,z) \rightarrow H(z)$$
with exponentially fast convergence: this completes the proof of Theorem (\ref{essesette}). The Hopf vector field on the divisor of the blow up is:
$$H(z) = A z$$
with $A$ made of four identical diagonal blocks:
\[
A=
\left(
\begin{array}{cccc}
R & 0 & 0 & 0  \\
0 & R & 0 & 0  \\
0 & 0 & R & 0 \\
0 & 0 & 0 & R 
\end{array} \right)
\]
where:
\[
R =
\left(
\begin{array}{cc}
0 & -1 \\
1 & 0
\end{array} \right).
\]
Therefore:
\[
\dot x = r \dot z = r A z + \Vert x 	\Vert Y ( \Vert x \Vert , \frac{x}{\Vert x \Vert})
\]
hence:
\[
\dot x = A x + \tilde Y(x)
\]
with $\tilde Y(x)$ a smooth exponentially flat vector field at the origin. The proof of the theorem is complete.

\end{proof}

We conclude making few remarks related to the results of this article.

Firstly, we point out that our example of non-linearizable nondegenerate multicentre is neither Hamiltonian (with respect to the standard symplectic stucture on $\mathbb{R}^8$) nor real analitic, and is therefore natural to wonder about an analogous example in these settings.

Moreover, our example differs from the linearized vector field at $0$ for an infinitely flat vector field, hence proving that no normal form argument can distinguish smooth nonlinearizable nondegenerate multicentres from the linearizables ones.

Finally, we observe that a simple application of Seifert's Stability Theorem \cite{sullivan}, \cite{bgv} implies that any vector field $Y$, sufficiently $C^2$-close to $X$ and still having $x^2_1 + \cdots + x^2_8$ as first integral, has at least one periodic orbit of minimal period close to $2 \pi$, contained in each level set of the first integral. Perhaps it would be intersting, if possible, to find some reasonable genericity assumptions on $Y$ in order that it has a sequence of periodic orbits accumulating at the stationary point, having unbounded minimal periods.

\end{document}